\documentclass[%
british,
final,
]{article}

\usepackage[utf8]{inputenc}
\usepackage[T1]{fontenc}
\usepackage{microtype}
\usepackage{titlesec}
\usepackage{babel}
\usepackage{csquotes}

\usepackage{amsmath}
\usepackage{amsthm}
\usepackage{amssymb}
\usepackage{amsfonts}
\usepackage{thmtools}
\usepackage{mathtools}
\usepackage{stmaryrd}
\usepackage{bbm} 

\usepackage{adjustbox}
\usepackage{tikz-cd}
    \usetikzlibrary{decorations.markings, arrows.meta, knots}
    \makeatletter
    \tikzcdset{
      open/.code     = {\tikzcdset{hook, circled};},
      closed/.code   = {\tikzcdset{hook, slashed};},
      open'/.code    = {\tikzcdset{hook', circled};},
      closed'/.code  = {\tikzcdset{hook', slashed};},
      circled/.code  = {\tikzcdset{markwith = {\draw (0,0) circle (.375ex);}};},
      slashed/.code  = {\tikzcdset{markwith = {\draw[-] (-.4ex,-.4ex) -- (.4ex,.4ex);}};},
      markwith/.code ={
        \pgfutil@ifundefined%
        {tikz@library@decorations.markings@loaded}%
        {\pgfutil@packageerror{tikz-cd}{You need to say %
          \string\usetikzlibrary{decorations.markings} to use arrows with markings}{}}{}%
        \pgfkeysalso{/tikz/postaction = {
          /tikz/decorate,
          /tikz/decoration={markings, mark = at position 0.5 with {#1}}}
        }
      },
    }
    \makeatother
\usetikzlibrary{tqft}

\RequirePackage[notref, notcite]{showkeys}
\RequirePackage[obeyFinal]{todonotes}

\RequirePackage{mathscinet}
\RequirePackage[backend      = biber,
                sortcites    = true,
                giveninits   = true,
                doi          = false,
                isbn         = false,
                url          = true,
                maxnames     = 5,
                maxcitenames = 5,
                minalphanames= 5,
                maxalphanames= 5,
                style        = alphabetic]{biblatex}
\DeclareNameAlias{sortname}{family-given}
\DeclareNameAlias{default}{family-given}
\DeclareFieldFormat[article]{volume}{\bibstring{jourvol}\addnbspace#1}
\DeclareFieldFormat[article]{number}{\bibstring{number}\addnbspace#1}
\renewbibmacro*{volume+number+eid}
{
    \printfield{volume}
    \setunit{\addcomma\space}
    \printfield{number}
    \setunit{\addcomma\space}
    \printfield{eid}
}
\renewbibmacro{in:}{}       
\DeclareFieldFormat{pages}{#1} 

\DeclareMathOperator{\A}{\mathbb{A}}
\DeclareMathOperator{\Z}{\mathbb{Z}}

\DeclareMathOperator{\R}{\mathbb{R}}

\DeclareMathOperator{\G}{\mathbb{G}}
\renewcommand{\P}{\mathbb{P}}
\renewcommand{\S}{\mathbb{S}}

\DeclareMathOperator{\Man}{\textnormal{\textbf{Man}}}

\DeclareMathOperator{\op}{\textnormal{op}}
\DeclareMathOperator{\Spc}{\textnormal{\textbf{Spc}}}
\DeclareMathOperator{\SH}{\mathcal{SH}}
\DeclareMathOperator{\Sm}{\textnormal{Sm}_S}
\DeclareMathOperator{\Spt}{\textnormal{\textbf{Spt}}} 
\DeclareMathOperator{\Cor}{\textnormal{\textbf{Cor}}}
\DeclareMathOperator{\Corf}{\Cor^{\textnormal{fét}}}
\DeclareMathOperator{\Shv}{\textnormal{Shv}}

\DeclareMathOperator{\Nis}{\textnormal{Nis}}

\DeclareMathOperator{\Spec}{\textnormal{Spec}}
\DeclareMathOperator{\Map}{\textnormal{Map}}
\DeclareMathOperator{\id}{\textnormal{id}}

\DeclareMathOperator{\Hom}{\textnormal{Hom}}
\DeclareMathOperator{\KGL}{\textnormal{KGL}}
\DeclareMathOperator{\from}{\colon}
\DeclareMathOperator{\der}{\textnormal{der}}
\DeclareMathOperator{\fr}{\textnormal{fr}}
\DeclareMathOperator{\fsyn}{\textnormal{fsyn}}

\DeclareMathOperator{\veff}{\textnormal{veff}}
\DeclareMathOperator{\gp}{\textnormal{gp}}

\declaretheoremstyle[headfont   = \bfseries\sffamily,
                     notefont   = \normalfont,
                     bodyfont   = \itshape,
                     spaceabove = 6pt,
                     spacebelow = 6pt]{plain}
\declaretheoremstyle[headfont   = \bfseries\sffamily,
                     notefont   = \normalfont,
                     spaceabove = 6pt,
                     spacebelow = 6pt]{definition}
\declaretheorem[style = plain, numberwithin = section]{theorem}
\declaretheorem[style = plain,      sibling = theorem]{corollary}

\declaretheorem[style = plain,      sibling = theorem]{proposition}

\declaretheorem[style = plain,      sibling = theorem]{construction}
\declaretheorem[style = definition, sibling = theorem]{definition}
\declaretheorem[style = definition, sibling = theorem]{example}
\declaretheorem[style = remark,     sibling = theorem]{remark}
\declaretheorem[style = plain,     sibling = theorem]{question}
\numberwithin{equation}{section}

\usepackage{varioref}
\usepackage[bookmarks=false]{hyperref}
\usepackage{cleveref}
\crefname{question}{Question}{Questions}

\title{Introduction to Framed Correspondences}
\author{Marc Hoyois and Nikolai Opdan}
\date{}

\begin{document}
\maketitle

\begin{abstract}
    We give an overview of the theory of framed correspondences in motivic homotopy theory. Motivic spaces with framed transfers are the analogue in motivic homotopy theory of \(E_{\infty}\)-spaces in classical homotopy theory, and in particular they provide an algebraic description of infinite \(\P^1\)-loop spaces. We will discuss the foundations of the theory (following Voevodsky, Garkusha, Panin, Ananyevskiy, and Neshitov), some applications such as the computations of the infinite loop spaces of the motivic sphere and of algebraic cobordism (following Elmanto, Hoyois, Khan, Sosnilo, and Yakerson), and some open problems.\footnote{This text is based on a series of three lectures given by the first author in August 2020 during the Motivic Geometry program at the Centre for Advanced Study at the Norwegian Academy of Science and Letters in Oslo, Norway.}
\end{abstract}

\tableofcontents

\section{Introduction}
Voevodsky introduced in an unpublished note (\cite{Framed}) framed correspondences to get a more computation-friendly model for motivic homotopy theory. In \cref{sec:Cobordism} we will see how this can be used in practice to get results for algebraic cobordism. First, we introduce framed correspondences as a tool to answer the following two fundamental questions.

\begin{question}\label{question1}
What kind of transfers do cohomology theories represented by motivic spectra have?
\end{question}

\begin{question}\label{question2}
Is every cohomology theory with these transfers represented by a motivic spectrum?
\end{question}

These questions will be answered in \cref{sec:Framed} and \cref{sec:Recognition} respectively.

We begin with an analysis of these problems in classical topology.

\begin{definition}
Let \(\Man\) denote the category of smooth manifolds and \(\Spc\) the \(\infty\)-category of spaces.
A \textit{cohomology theory} on \(\Man\) is a functor 
\[F \from \Man^{\op} \longrightarrow \Spc\]
satisfying
\begin{enumerate}
    \item Descent with respect to arbitrary open coverings.
    \item Homotopy invariance: \(F(M) \overset{\simeq}{\to} F(M \times \R)\) for all \(M\in\Man\).
\end{enumerate}
\end{definition}

We have the following classification.
\begin{theorem}
\label{thm:TopTransfers}
The evaluation functor
\[\begin{array}{rcl}
\{\text{Coh. theories on } \Man\} & \overset{\simeq}{\longrightarrow} &\Spc \\
                            F & \longmapsto     & F(*)
\end{array}\]
is an equivalence.
\end{theorem}

\begin{proof}
The constant sheaf \(\underline X\) on \(\Man\) with fibre a space \(X\) is given by
\[M \mapsto \Map(M, X).\]
(This is a nontrivial but well-known computation, which uses the fact that manifolds are sufficiently nice topological spaces, in particular locally contractible.) This computation shows that \(\underline X\) is homotopy invariant. In particular, \(X\mapsto \underline X\) is the left adjoint to the given evaluation functor, and since \(\underline X(*)=X\) it is fully faithful. To show that it is an equivalence of categories it remains to show that evaluation on the point is conservative. 

Let \(f\) be a morphism between cohomology theories and supposes that it is an isomorphism on a point. We need to show that it is an isomorphism in general. By homotopy invariance, we see that \(f\) is an isomorphism on the Euclidian spaces \(\R^n\) for all \(n\), and by using the descent property on good covers (covers are copies of \(\R^n\) such that all intersections are either empty or isomorphic to \(\R^m\)) we can deduce that it is indeed an isomorphism for all manifolds \(M\).
\end{proof}

In light of \Cref{thm:TopTransfers}, cohomology theories can be described by giving a space \(X\).
If \(X\) happens to be the infinite loop space of a spectrum \(E\), then the
associated cohomology theory 
\[M \mapsto \Map(M, \Omega^{\infty}E)\]
acquires some extra structure.

\begin{example}[{Atiyah duality}]
If \(M\) is compact, then \(\Sigma_+^{\infty} M\) has as dual the Thom spectrum \(M^{-T_M}\), where \(T_M\) is the tangent bundle of \(M\). This implies that for every morphism \(f \from M \to N\) between compact manifolds there is an induced morphism of Thom spectra \(M^{-T_M} \leftarrow N^{-T_N}\). Mapping these spectra into another spectra \(E\) we get a covariant pushforward map \todo{Is this the equation the correct way to express this?}
\[\Map(M^{-T_M}, E) \overset{}{\longrightarrow} \Map(N^{-T_N}, E).\]
If \(T_M\) and \(T_N\) are oriented with respect to \(E\) (e.g. if \(E\) is complex oriented and \(M\) and \(N\) are complex manifolds), then the pushforward involves ordinary shifts. 
\end{example}

One can more generally define a cohomological pushforward along any proper morphism \(f \from M \to N\) between (not necessarily compact) smooth manifolds. This is a consequence of the ``formalism of six operations'', which is a vast generalization of Poincaré/Atiyah duality introduced by Grothendieck. For any morphism \(f \from M \to N\) in \(\Man\), there is the usual pullback/pushforward adjunction
\[
\begin{tikzcd}[column sep=large]
\Shv(M, \Spt) \arrow[r, "f_*"{below}, shift right=3pt]
& \Shv(N, \Spt), \arrow[l, "f^*"{above}, shift right=3pt]
\end{tikzcd}
\]
as well as the ``exceptional'' adjunction
\[
\begin{tikzcd}[column sep=large]
\Shv(M, \Spt) \arrow[r, "f_!"{above}, shift left=3pt]
& \Shv(N, \Spt). \arrow[l, "f^!"{below}, shift left=3pt]
\end{tikzcd}
\]
Here, \(f_!\) is the pushforward with compact support, which is essentially determined by the following two properties:
\begin{itemize}
\item If \(f\) is proper, then \(f_!=f_*\).
\item If \(f\) is étale (i.e., a local homeomorphism), then \(f^!=f^*\).
\end{itemize}
Defining the virtual tangent bundle of \(f\) by 
\[T_f:= T_M - f^*T_N,\] 
there exists then a canonical natural transformation 
\[\mathfrak p_f\from \Sigma^{T_f} f^* \to f^!,\]
where \(\Sigma^{T_f}\) denotes suspension by the virtual tangent bundle, i.e., smashing with its Thom spectrum. This functor is an automorphism of \(\Shv(M, \Spt)\), and locally on \(M\), it is the ordinary suspension by the rank of \(T_f\). We do not explain here the construction of \(\mathfrak p_f\). Suffice it to say, because any morphism factors as the composition of a submersion and a closed immersion and because of the tubular neighborhood theorem, the transformation \(\mathfrak p_f\) is essentially determined by the cases where \(f\) is a submersion (in which case it is an \emph{equivalence} \(\Sigma^{T_f}f^*\simeq f^!\)), and where \(f\) is the zero section of a vector bundle. We will explain later the algebro-geometric analogue of \(\mathfrak p_f\) in more details.

If \(E_M\) denotes the constant sheaf on \(M\) with value \(E\), we obtain morphisms
\[\Gamma(M, \Sigma^{T_f}E_M) \simeq \Gamma(M, \Sigma^{T_f}f^*E_N) \xrightarrow{\mathfrak p_f} \Gamma(M, f^!E_N) \simeq \Map(\S_M, f^!E_N).\]
If \(f\) is proper we furthermore have morphisms
\[\Map(\S_M, f^!E_N) \simeq \Map(f_*f^*(\S_N), E_N) \rightarrow
\Map(\S_N, E_N) = \Gamma(N, E_N).\]
Composing these morphisms gives us a pushforward map
\[\Gamma(M, \Sigma^{T_f}E_M) \to \Gamma(N, E_N)\]
also known as the \textit{transfer} along \(f\). This provides an answer to \cref{question1} in the classical topological context.

For \cref{question2}, the question becomes whether one can recover the spectrum structure on \(E\) from the infinite delooping \(\Omega^{\infty}E\) by using these transfers? 
In particular, if \(f \from M \to N\) is a finite étale map (i.e., a finite covering map), then \(T_f = 0\), in which case there is a canonical transfer 
\begin{equation*}
    \Map(M, \Omega^{\infty}E) \to \Map(N, \Omega^{\infty}E).
\end{equation*}

By taking \(N\) to be a point and \(M\) to be two distinct points, the transfer provides an addition map
\begin{equation*}\label{eq:Coherentness}
    \Omega^{\infty}E \times \Omega^{\infty} E \overset{+}{\longrightarrow} \Omega^{\infty}E.
\end{equation*}
We need to encode these transfers in a suitably coherent manner to recover the commutativity of $+$ up to coherent homotopy. Historically, this has been difficult because one lacked the tools to precisely express the required coherence. To achieve this we introduce the 2-category of correspondences:

\begin{definition}
The \textit{category of finite étale correspondences} in \(\Man\), denoted by \(\Corf(\Man)\), is the 2-category where objects are smooth manifolds, and morphisms from \(M\) to \(P\) are spans
\[\begin{tikzcd}
& N \arrow[dl, "\text{f.ét.}"{description}] \arrow[dr] \\
M & & P,
\end{tikzcd}\]
such that the map \(N \to M\) is finite étale. These spans are known as \textit{correspondences}.
Composition of two correspondences \((M \leftarrow N \rightarrow P)\) and \((P \leftarrow Q \rightarrow R)\) is defined as the pullback in the diagram
\[\begin{tikzcd}[column sep = small, row sep = small]
& & \text{pullback} \arrow[dl] \arrow[dr] & & \\ 
& N \arrow[dl] \arrow[dr] & & Q \arrow[dl] \arrow[dr] \\
M & & P & & R.
\end{tikzcd}\]
The pullback exists in \(\Man\) because the map \(Q \to P\) is finite étale.
\end{definition}

The following theorem then provides an answer to \cref{question2}.
\begin{theorem}\label{thm:Q1}
There exists a functor
\[\begin{array}{rcl}
\Spt  & \longrightarrow & \{\text{Coh. theories on } \Corf(\Man)\} \\
E & \longmapsto     & \Map( - , \Omega^{\infty}E),
\end{array}\]
which restricts to an equivalence between connective spectra \(\Spt_{\geq 0}\) and grouplike cohomology theories on \(\Corf(\Man)\).
\end{theorem}
Such cohomology theories are automatically evaluated in the \(\infty\)-category of \(E_\infty\)-spaces, hence the ``grouplike'' condition means that the addition structure
\[\Omega^{\infty}E \times \Omega^{\infty}E  \to \Omega^{\infty}E \]
should be a group, i.e., elements should have inverses up to homotopy.
We refer to \cite[Appendix C]{BH} for a proof of Theorem~\ref{thm:Q1} from Theorem~\ref{thm:TopTransfers}.

\subsection*{Acknowledgements}
The authors greatly acknowledge the support from the research project  ``Motivic Geometry'' at the Centre for Advanced Study at the Norwegian Academy of Science and Letters in Oslo, Norway.

\section{Transfers for motivic spectra}\label{sec:MotivicTransfers}
We will now focus on \cref{question1}, i.e., transfers for motivic spectra.

Let \(\SH(S)\) denote the \(\infty\)-category of motivic spectra over \(S\). There is a six functor formalism on the system of categories \(S \mapsto \SH(S)\), i.e.,
functors \[f^*, f_*, f^!, f_!, \underline{\Hom}, \otimes,\] with adjuctions \(f^* \dashv f_*\) and \(f_! \dashv f^!\), and equivalences \(f^! \simeq f^*\) if \(f\) is étale and \(f_! \simeq f_*\) if \(f\) is proper. 

For the case of smooth manifolds, we saw in the previous section that there are transfers for any proper morphism. A key difference in the algebraic-geometric context is that 
we are now also willing to consider schemes which are not smooth. This changes the picture since morphisms of smooth manifolds are very far from being arbitrary morphisms, and in particular they are always of local complete intersections. The condition of local complete intersection turns out to be the most general condition for transfers to exist in this setting.

\begin{definition}
A morphism of schemes \(f \from X \to S\) is a \textit{local complete intersection} (abbreviated ``lci'') if locally on \(X\) it factors as 
\begin{center}
    \begin{tikzcd}
    X \arrow[r, "i", hook, closed] \arrow[dr, "f"] & V \arrow[d, "p"] \\
    & S,
    \end{tikzcd}
\end{center}
where \(i\) is closed immersion which locally is cut out by a regular sequence (i.e., a \textit{regular} closed immersion) and \(p\) is a smooth morphism.
\end{definition}

For any local complete intersection morphism \(f \from X \to S\) with a global factorization as above, we can define the \textit{virtual tangent bundle} of \(f\) by
\[T_f := i^*T_p - N_i \in K(X).\]
Here, \(T_p\) is the relative tangent bundle of \(p\), \(N_i\) is the normal bundle of \(i\), and \(K(X)\) denotes the \(K\)-theory space of the scheme \(X\).
One can show that \(T_f\) does not depend on the factorization of \(f\) (it is in fact the image in \(K(X)\) of the tangent complex of \(f\), which is defined even if no factorization exists).

We will now sketch the fundamental construction of Déglise, Jin, and Khan \cite{DJK}, which is the source of the transfers in stable motivic homotopy theory. We recall that any element \(\xi \in K(X)\) has an associated Thom spectrum in \(\SH(X)\), and smashing with it defines a self-equivalence
\[\Sigma^{\xi} \from \SH(X) \overset{\simeq}{\longrightarrow} \SH(X).\]

\begin{construction}[\cite{DJK}]
For an lci morphism \(f \from X \to S\) with a global factorization, we construct a canonical transformation
\[\mathfrak p_f\from\Sigma^{T_f}f^* \longrightarrow f^!.\]
\end{construction}
\begin{proof}[Sketch]
Because of the global lci factorization, it suffices to consider the case of smooth morphisms and regular closed immersions. 

If \(f\) is smooth we let \(\mathfrak p_f\) be the well-known purity equivalence
\[\Sigma^{T_f}f^* \simeq f^!\]
due to Voevodsky and Ayoub.

If \(f=i \from X \hookrightarrow V\) is a regular closed immersion we consider the deformation of \(V\) to the normal cone, which is a scheme \(D_X V\) that fits in a diagram
\begin{center}
    \begin{tikzcd}
        & N_i \arrow[r, hook, closed, "t"] \arrow[dd] \arrow[dl, "\pi", twoheadrightarrow] & D_X V \arrow[dd] & \G_m \times V \arrow[l, hook', open', "u" above] \arrow[dd, equal] \\
        X \arrow[dr, hook, "i"] \\
        & V \arrow[r, hook, closed, "0" above] \arrow[dr, equal] & \A^1_V \arrow[d] & \G_m \times V \arrow[l, hook', open'] \arrow[dl, "q"] \\
        & & V,
    \end{tikzcd}
\end{center}
where \(0\from V \hookrightarrow \A^1_V\) is the zero section, \(\pi\from N_i \to X\) is the normal bundle of \(i\), \(\G_m \times V \hookrightarrow D_XV\) and \(\G_m \times V \hookrightarrow \A^1_V\) are the open complements, and both squares are cartesian\footnote{For an arbitrary closed immersion \(i\), the fibre of \(D_XV\) over \(0\) is the normal cone of \(i\), which agrees with the normal bundle in case of a regular immersion.}. Let \(p\) denote the composition of the maps \(D_XV \to \A^1_V \to V\).

The top row gives rise to a localization triangle
\[t_*t^! \to \id_{\SH(D_XV)} \to u_*u^! \overset{\partial}{\to} t_*t^![1],\]
and in \(\SH(\G_m)\) there is a canonical (universal unit) map \(S^0 \to \G_m\). 
Suspending this map with \(S^1\) we get a new map 
\[S^1 \overset{(*)}{\longrightarrow} S^1 \wedge \G_m = T\]
where \(T = S^1 \wedge \G_m\) is the Thom space of the trivial line bundle. We want to define a natural transformation 
\[\mathfrak p_i\from i^* \longrightarrow \Sigma^{N_i}i^!,\]
or equivalently, by adjunction, a map
\[\id \to i_*\Sigma^{N_i}i^!.\]
This we obtain from the composition
\[\id_{\SH(V)}[1] \overset{(*)}{\to} q_* \Sigma_T q^* \simeq q_*q^! \simeq p_*u_*u^!p^! \overset{\partial}{\to} p_*t_*t^!p^![1] \simeq i_*\pi_* \pi^! i^![1],\]
and using the purity equivalence \(\mathfrak p_\pi\from \pi^*\Sigma^{N_i}\simeq \pi^!\) and the fact that \(\pi^*\) is a fully faithful functor, we further get 
\[i_*\pi_* \pi^! i^! \simeq i_*\Sigma^{N_i} i^!.\]
Composing these maps gives the desired natural transformation
\[\id_{\SH(V)} \to i_*\Sigma^{N_i} i^!.\]

It remains to prove that this procedure is independent of the factorization, composing various \(f\)'s, etc. We refer the curious reader to \cite{DJK} for a thorough treatment of these facts. 
\end{proof}

\begin{definition}
Let \(E \in \SH(S)\) be a motivic spectrum over \(S\).
For a morphism \(f \from X \to S\) and an element \(\xi\in K(X)\) we define
\[E(X, \xi) := \text{Map}_{\SH(X)}(\mathbbm{1}_X, \Sigma^{\xi}f^*E) \in \Spc.\]
\end{definition}

\begin{example}
If \(\xi = n\), then \(\Sigma^{\xi} \cong \Sigma^{2n, n}\). Thus
\[\pi_i E(X, \xi) = E^{2n-i, n}(X).\]
\end{example}

\begin{example}
If \(E = \KGL\), then \(E(X, \xi)\cong \text{KH}(X)\), where KH denotes the homotopy K-theory of \(X\). Note that this does not depend on \(\xi\). 
\end{example}

\begin{example}
If \(E = H\mathbb{Z}\) and \(X\) is smooth over a field or a discrete valuation ring, then \(E(X, \xi) \cong z^r(X, \bullet)\) where \(r = \text{rank}(\xi)\). Here \(z^r(X, \bullet)\) denotes Bloch's cycle complex.
\end{example}

\begin{theorem}\label{thm:Transfers}
If \(f \from X \to S\) is proper and lci, the natural transformation \(\mathfrak p_f\from\Sigma^{T_f}f^* \to f^!\) induces a transfer map
\[E(X, T_f + f^* \xi) \longrightarrow E(S, \xi),\]
where \(\xi \in K(S)\).
\end{theorem}

\begin{remark}
More generally, if \(f \from X \to S\) is proper and \(X^{\der}\) is a derived structure on \(X\) which is lci/quasi-smooth over \(S\), a theorem by Adeel A. Khan (\cite{Kha}) gives a transfer map
\[E(X, T_f^{\der}) \longrightarrow E(S).\]
\end{remark}

\section{The \(\infty\)-category of framed correspondences}\label{sec:Framed}
Following the approach to \cref{question1} in classical topology, we are interested in those transfers that do not shift the degree, i.e., that induce a covariant functoriality on the ``unshifted'' cohomology theory. For this to work we saw in the previous section that we need the virtual tangent bundle of a map \(f\) to vanish. We are thus led to the following definition:

\begin{definition}
Suppose \(f \from X \to S\) is a lci morphism. A (stable 0-dimensional) \textit{framing} of \(f\) is a trivialization of the virtual tangent bundle \(T_f\), 
i.e., a path in \(K(X)\) between \(T_f\) and \(0\).
\end{definition}

If \(f\) is a framed proper lci morphism, we see by \Cref{thm:Transfers} that it induces a transfer map in cohomology that does not shift the degree. This provides an answer to \cref{question1}.

\begin{definition}
We define the \(\infty\)-category \(\Cor^{\fr}(\Sm)\) with objects smooth \(S\)-schemes and morphisms given as spans \((X \overset{f}{\leftarrow} Z \rightarrow Y)\),
where \(f\) is finite lci and framed (in the sense that there is given a path \(\alpha \from T_f \overset{\simeq}{\longrightarrow} 0\) in \(K(Z)\)).

Composition of morphisms \((X \leftarrow Z \rightarrow Y)\) and \((Y \leftarrow W \rightarrow V)\)
is defined by pullback
\begin{center}
    \begin{tikzcd}[column sep=tiny, row sep=tiny]
    & & \text{pullback} \arrow[dl, "h"] \arrow[dr, "b"] & & \\
    & Z \arrow[dl, "f"] \arrow[dr, "a"] & & W \arrow[dl, "g"] \arrow[dr] & \\
    X & & Y & & V,
    \end{tikzcd}
\end{center}
where \(T_{f \circ h} \simeq T_h + h^* T_f\), \(T_h \simeq b^*(T_g)\simeq 0\) and \(h^*T_f \simeq 0\).
\end{definition}

\begin{remark}
The condition for a morphism to be finite lci and framed implies that it is flat. This is relevent because it forces the condition of being lci to be stable under base change. In general the lci condition is only stable under tor-independent base change. 
\end{remark}

\begin{remark}
The reason why we only consider finite maps is that we want the framed transfers to be compatible with the Nisnevich topology. That is, we want the Nisnevich sheafication of a presheaf with framed transfers to have framed transfers, and this requires finite morphisms. It is possible to relax the finiteness condition by instead considering proper morphisms, see \cref{OpenQuestion}. However, we must then instead consider derived schemes since such morphisms are not flat.
\end{remark}

\begin{remark}
The category \(\Cor^{\fr}(\Sm)\) is semi-additive (i.e., finite sums and finite products coincide) and has a symmetric monoidal structure given by the Cartesian products (\(X \otimes Y = X \times_S Y\)). Moreover, there are canonical functors
\[\Corf(\Sm) \longrightarrow \Cor^{\fr}(\Sm),\]
because finite étale morphisms have a canonical framing. By forgetting the framing we moreover have a functor
\[\Cor^{\fr}(\Sm) \longrightarrow \Cor^{\fsyn}(\Sm)\]
to the category of finite syntomic (flat and lci) correspondences. A further forgetful functor takes us to Voevodsky's category of correspondences \(\Cor(\Sm)\).
\end{remark}

\section{The recognition principle}\label{sec:Recognition}
We here mention three fundamental theorems, due to work by Garkusha, Panin, Ananyevskiy, and Neshitov, and Elmanto, Hoyois, Khan, Sosnilo, and Yakerson. These will allow us to prove a recognition principle for framed correspondences that answers \cref{question2}.

\begin{definition}
Let \(\mathcal{H}(S)\) denote the \(\infty\)-category of \(\A^1\)-invariant Nisnevich sheaves on \(\Sm\), and let \(\mathcal{H}^{\fr}(S)\)
be the \(\infty\)-category of \(\A^1\)-invariant Nisnevich sheaves on \(\Cor^{\fr}(\Sm)\) (i.e., presheaves on \(\Cor^{\fr}(\Sm)\) whose restriction to \(\Sm\) are \(\A^1\)-invariant Nisnevich sheaves).
We define \(\SH(S)\) and \(\SH^{\fr}(S)\) as the \(\infty\)-categories of \(T\)-spectra in \(\mathcal{H}(S)\) and \(\mathcal{H}^{\fr}(S)\), respectively.
\end{definition}

The functor \(\Sm \to \Cor^{\fr}(\Sm)\) induces adjoint functor pairs
\begin{center}
    \begin{tikzcd}
    \mathcal{H}(S) \arrow[r, shift left=2pt, "\text{free}"] \arrow[d, shift right=2pt, "\Sigma^\infty_+" left] & \mathcal{H}^{\fr}(S) \arrow[l, shift left=2pt, "\text{forget}"] \arrow[d, shift right=2pt, "\Sigma^\infty_{\fr}" left] \\
    \SH(S) \arrow[r, shift left=2pt, "\text{free}"] \arrow[u, shift right=2pt, "\Omega^{\infty}" right]& \SH^{\fr}(S) \arrow[l, shift left=2pt, "\text{forget}"] \arrow[u, shift right=2pt, "\Omega^{\infty}_{\fr}" right] .
    \end{tikzcd}
\end{center}

\begin{theorem}[Reconstruction theorem, {\cite[Theorem 16]{Hoy}}]\label{thm:1}
For every scheme \(S\), the canonical functor
\[\SH(S) \overset{\sim}{\longrightarrow} \SH^{\fr}(S)\]
is an equivalence of \(\infty\)-categories.

Thus every motivic spectrum has a unique structure of framed transfers.
\end{theorem}

This result requires new ideas since the analogue statement is false for Voevodsky's correspondences because it would yield an equivalence \(\SH^{\fr}(S) \simeq \textbf{DM}(S)\). 

We provide a brief sketch of the main ideas involved in the proof.
\begin{proof}[Sketch of proof]
The key idea is to apply Voevodsky's lemma:
That is, if \(X\) is a scheme, and \(U \subset X\) is an open subscheme, we can identify \(L_{\Nis}(X/U)(Y)\) with the set of pairs \((Z, \phi)\) such that \(Z \subset Y\) is a closed subset and \(\phi \from Y^h_Z \to X\) is a map such that \(\phi^{-1}(X \setminus U)=Z\)

Then consider the following mapping space for \(X, Y \in \Sm\) together with the equivalence:
\[\text{Map}\left((\P^1)^{\wedge n} \wedge X_+, L_{\Nis}\left( \A^n/\A^n - 0\right) \wedge Y_+ \right) \cong (\Omega^n_{\P^1}L_{\Nis}\Sigma^n_T Y_+)(X).\]
Applying Voevodsky's lemma we find that this equals the set of pairs \((Z, \phi)\) such that 
\begin{enumerate}
    \item \(Z \subset (\P^1_X)^n\) is a closed subscheme,
    \item \(Z \cap \partial (\P^1_X)^n = \emptyset\), 
    \item \(\phi \colon ((\P_X^1)^n)^h_Z \to \A^1_Y\) such that \(\phi^{-1}(0) = Z.\)
\end{enumerate}
Thus \(Z\) is cut out by \(n\) equations with codimension \(n\), and in particular, it is framed. This shows that there is a forgetful map to \(\Cor^{\fr}_S(X, Y)\). We conclude by \cite[Corollary 2.3.27]{EHKSY} which says that the forgetful map
\[\underset{n \to \infty}{\text{colim }} \Omega^n_{\P^1}L_{\Nis}\Sigma^n_T Y_+ \longrightarrow \Cor_S^{\fr}(-,Y)\]
is a motivic equivalence if \(Y \in \Sm\).
\end{proof}

\begin{theorem}[Cancellation theorem, {\cite[Proposition 5.2.7]{EHKSY}}]\label{thm:2}
Let \(k\) be a perfect field. The functors 
\[S^1 \wedge (-),\; \G_m \wedge (-) \colon \mathcal{H}^{\fr}(k)^{\gp} \longrightarrow \mathcal{H}^{\fr}(k)^{\gp}\]
are fully faithful, where \(\mathcal{H}^{\fr}(k)^{\gp}\) is the full subcategory of \(\mathcal{H}^{\fr}(k)\) consisting of grouplike objects.
\end{theorem}

Let \(L_{\mathrm{Zar}}\) denote the Zariski sheafication functor and \(L_{\A^1}\) the functor defined by 
\[(L_{\A^1}F)(U) = \vert F(U \times \A^{\bullet})\vert.\]

\begin{theorem}[Strict \(\A^1\)-invariance theorem, {\cite[Theorem 3.4.11]{EHKSY}}]\label{thm:3}
Let \(k\) be a perfect field. If \(F \in \mathcal{P}_{\Sigma}(\Cor^{\fr}(\mathrm{Sm}_k))^{\gp}\), then 
\(L_{\mathrm{Zar}}L_{\A^1}F\)
is an \(\A^1\)-invariant Nisnevich sheaf.
\end{theorem}

\Cref{thm:2} and \Cref{thm:3} are analogues of Voevodsky's results for finite correspondences (\cite{Cancellation} and \cite[§3.2]{StrictInvariance}) and the proofs are similar. For \Cref{thm:2}, Bachmann has a paper (\cite{Tom}) where he proves a cancellation theorem for finite flat correspondences, which applies to these framed correspondences as well. 

The next result is an analogue of \Cref{thm:Q1} for framed motivic spectra and thus answers \cref{question2}. 
\begin{corollary}[Motivic recognition principle, {\cite[Theorem 3.5.16]{EHKSY}}]\label{cor:Q1}
Let \(k\) be a perfect field. There is an equivalence
\[\mathcal{H}^{\fr}(k)^{\gp} \simeq \SH(k)^{\veff},\]
where \(\SH(k)^{\veff} \subset \SH(k)\) denotes the full subcategory generated under colimits by suspension spectra \(\Sigma^{\infty}_+ X.\)
\end{corollary}
\begin{proof}
\Cref{thm:2} implies that the infinite suspension functor 
\[\Sigma^{\infty}_{\fr} \from \mathcal{H}^{\fr}(k)^{\gp} \to \SH^{\fr}(k)\]
is fully faithful. Using \Cref{thm:1} we identify its essential image as \(\SH(k)^{\veff}.\)
\end{proof}

The main application of \Cref{thm:1} is that the framed suspension functor \(\Sigma^{\infty}_{\fr}\) becomes a ``machine'' for producing motivic spectra:
Let \(F \in \mathcal{P}_{\Sigma}(\Cor^{\fr}(\Sm))\) be a presheaf with framed transfers. Then 
\[\Sigma^{\infty}_{\fr} F \in \SH^{\fr}(S),\]
which by \Cref{thm:1} can by identified with an object of \(\SH(S).\) This allows one to build motivic spectra from presheaves with framed transfers.

In the special case of \(S= \Spec k\), where \(k\) is a perfect field, we can use \cref{cor:Q1} to compute the infinite loop space of such a spectrum by using the equivalence
\[\Omega^{\infty}_{\fr}\Sigma^{\infty}_{\fr}F \simeq L_{\mathrm{Zar}}L_{\A^1} F^{\gp}.\]

We have the following consequences of this in action:
\begin{example}
Applying the infinite framed suspension functor \(\Sigma^{\infty}_{\fr}\) to the right side of this diagram one obtains the left side:
\begin{center}
    \begin{tikzcd}
    \mathbbm{1} \arrow[d] & & & F Syn^{\fr}\arrow[d]
    \\
    \textnormal{MSL} \arrow[r] \arrow[d] & \textnormal{MGL} \arrow[d] && F Syn^{\text{or}} \arrow[hook,d] \arrow[r] & F Syn \arrow[hook,d]\\
    ko \arrow[r] \arrow[equal,d] & kgl \arrow[equal,d] && F Gor^{\text{or}} \arrow[r] \arrow[d] & F Flat \arrow[d]\\
    ko \arrow[r] \arrow[d] & kgl \arrow[d] && Vect^{\text{sym}} \arrow[r] \arrow[d] & Vect \arrow[d]\\
    H\widetilde{\Z} \arrow[r]& H\mathbb{Z} && \underline{GW} \arrow[r] & \underline\Z,
    \end{tikzcd}
\end{center}
In more details:
\begin{itemize}
    \item \(FSyn^{\fr}\) is the stack of framed finite syntomic schemes, which is the unit object in \(\mathcal{P}_{\Sigma}(\Cor^{\fr}(\mathrm{Sm}_k))\). The claim that \(\Sigma^\infty_{\fr} FSyn^{\fr}\simeq \mathbbm{1}\) is therefore tautological.
    \item \(FSyn\) is the stack of finite syntomic schemes and \(FSyn^{\text{or}}\) is the stack of finite syntomic schemes with an orientation (i.e. a trivialization of the dualizing sheaf). That these give \(\textnormal{MGL}\) and \(\textnormal{MSL}\) over a general base is proved in \cite[Theorems 3.4.1, 3.4.3]{EHKSYCobordism}.
    \item \(Vect\) is the stack of vector bundles and \(FFlat\) is the stack of finite flat schemes. That both give the effective \(K\)-theory spectrum \(kgl\) over a field is proved in \cite[Corollary 5.2, Theorem 5.4]{HJNTY}.
    \item \(Vect^{\text{sym}}\) is the stack of vector bundles with a non-degenerate symmetric bilinear form, and \(FGor^{\text{or}}\) is the stack of oriented finite Gorenstein (dualizing sheaf is a line bundle) schemes. That both give the effective hermitian \(K\)-theory spectrum \(ko\) over a field of characteristic not 2 is proved in \cite[Proposition 7.7, Theorem 7.12]{HJNY}.
    \item \(\underline\Z\) is the constant sheaf with fiber \(\Z\). It is proved in \cite[Theorem 21]{Hoy} that \(\Sigma^\infty_{\fr} \underline\Z\simeq H\mathbb{Z}\) over any Dedekind domain, where \(H\mathbb{Z}\) is the motivic spectrum representing Bloch--Levine motivic cohomology.
    \item Finally, \(\underline{GW}\) is the sheaf of unramified Grothendieck--Witt rings, i.e., the Zariski sheafification of the usual Grothendieck--Witt groups. It is proved in \cite[Theorem 8.3]{HJNY} that \(\Sigma^\infty_{\fr} \underline{GW}\simeq H\widetilde{\Z}\) over any Dedekind domain in which 2 is invertible.
\end{itemize}
Here is one nice application of this picture: all morphisms on the right-hand side are easily seen to be \(E_\infty\)-ring maps, hence so are all morphisms on the left-hand side. 
\end{example}

\section{Algebraic cobordism}\label{sec:Cobordism}
The topic for this section is algebraic cobordism, following \cite{EHKSYCobordism}. Specifically, we will apply the above theory to provide a framed model for algebraic cobordism which allows for more friendly computations. 

\begin{quote}
    \textit{In particular the question about the convergence of the divisibility spectral sequence for algebraic cobordism would benefit greatly from knowing that elements of algebraic cobordism can be represented by some kind of geometric objects.}
    
    -V. Voevodsky, 2000, My view of the current state of motivic homotopy theory p.11,
\end{quote}

We begin by considering the moduli stack \(PQSm\) of proper quasi-smooth derived schemes. It is natural to study this stack in connection with algebraic cobordism, because these are precisely the schemes that give rise to algebraic cobordism classes. Namely, a proper quasi-smooth derived scheme of virtual dimension \(-n\) has a cobordism class in \(\textnormal{MGL}^{2n, n}\) by, for instance, using the transfers from \cref{sec:MotivicTransfers}. The \(\A^1\)-localization
\[L_{\A^1} PQSm := \vert PQSm({-} \times \A^\bullet) \vert\]
can be described informally as follows. Points in \(L_{\A^1} PQSm\) are schemes, and a path between two points is a scheme over \(\A^1\), which may be viewed as a cobordism from the fiber over \(0\) to the fiber over \(1\) (see \cref{fig:Cobordism}). A path between two paths is then a cobordism of cobordisms, etc. The analogous construction in topology, with \(\R^1\) in place of \(\A^1\), is one way to define the cobordism spaces that appear for example in the cobordism hypothesis. One may thus interpret the functor \(L_{\A^1}\) as forming the cobordism spaces of moduli stacks of schemes.

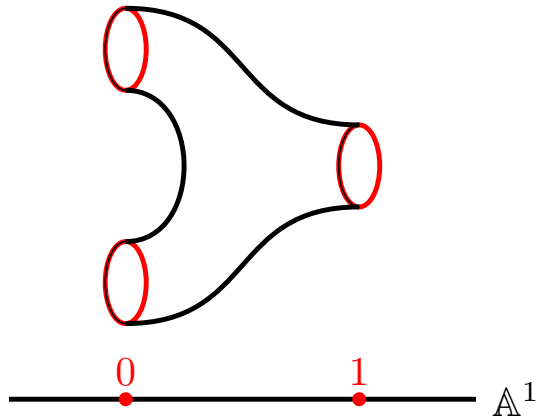
\begin{figure}[h]
    \centering
    \adjustbox{width=0.6\textwidth}{%
    \begin{tikzpicture}[tqft/view from=incoming]
    \pic[tqft/pair of pants, draw,at={(3,0)}, rotate=270, 
    cobordism edge/.style={draw,very thick, black},
    incoming boundary component 1/.style={draw,very thick, red, rotate=270},
    outgoing boundary component 1/.style={draw,very thick, red, rotate=270}, 
    outgoing boundary component 2/.style={draw,very thick, red, rotate=270}];
    
    \draw[black, very thick] (0,-2) -- (4,-2) node[anchor=west]{\(\mathbb{A}^1\)};
    \filldraw [red] (1,-2) circle (1.5pt) node[anchor=south]{0};
    \filldraw [red] (3,-2) circle (1.5pt) node[anchor=south]{1};
    \end{tikzpicture}}
    \caption{A cobordism from two circles to one.}
    \label{fig:Cobordism}
\end{figure}

Recall the forgetful functor
\[\Cor^{\fr}(\Sm) \longrightarrow \Cor^{\fsyn}(\Sm)\]
from the \(\infty\)-category of framed correspondences of smooth schemes to the 2-category of finite syntomic correspondences of smooth schemes, given by forgetting the framing and using that the morphism in the left span is finite syntomic. 

Let \(FQSm^n\) denote the moduli stack of finite quasi-smooth derived \(S\)-schemes of relative dimension \(-n\). This is obviously a presheaf on \(\Cor^{\fsyn}(\Sm)\) and hence also on \(\Cor^{\fr}(\Sm)\), using the above forgetful functor.

\begin{theorem}\label{thm:A1}
For every scheme \(S\) and every integer \(n\geq 0\) there is an equivalence
\[\Sigma^n_T \textnormal{MGL} \simeq \Sigma^{\infty}_{\fr} FQSm^n\]
in \(\SH(S) \simeq \SH^{\fr}(S)\).
\end{theorem}

The main ingredient of the proof is to understand the framed models for Thom spectra of virtual vector bundles coming from \Cref{thm:FramedThomSpectra}. We therefore suspend the proof of this theorem until we have established a proof of the result.


\begin{remark} There is an equivalence
\[FQSm^0 \cong FSyn.\]
\end{remark}

The next theorem is a specialization of \Cref{thm:A1} to the case of perfect fields. 
\begin{theorem}\label{thm:A2}
If \(k\) is a perfect field and \(n\geq 0\), then 
\[\Omega^{\infty}_T\Sigma^{n}_T \textnormal{MGL} \simeq L_{\mathrm{Zar}}L_{\A^1} (FQSm^n)^{\gp}.\]
In particular,
\[\Omega^{\infty}_T \textnormal{MGL} \simeq L_{\mathrm{Zar}}L_{\A^1} FSyn^{\gp}.\]
If \(n \geq 1\), then
\[\Omega^{\infty} \Sigma^{n}_T \textnormal{MGL} \simeq L_{\mathrm{Nis}}L_{\A^1} FQSm^n.\]
\end{theorem}

\begin{remark}
The last point of \Cref{thm:A2} means that one can trade the group completion by instead taking Nisnevich sheafification. More precisely, the Nisnevich sheaf \(L_{\mathrm{Nis}}L_{\A^1} FQSm^n\) is connected, hence is already grouplike, so we do not need to take its group completion. 

Moreover, a posteriori, we know by Morel's connectivity theorem that 
\(\Omega^{\infty} \Sigma^{n}_T \textnormal{MGL}\) is \((n-1)\)-connected.
\end{remark}

\begin{remark}
We state \Cref{thm:A1} and \Cref{thm:A2} for \(\textnormal{MGL}\), but there are analogue statement for all Thom spectra that one would like to think about, e.g. \(\textnormal{MSL}\), \(\textnormal{MSp}\), \(\mathbb{S}\), etc.  
\end{remark}

\begin{theorem}\label{thm:B1}
For every scheme \(S\) there is an equivalence of symmetric monoidal \(\infty\)-categories
\[\textnormal{Mod}_{\textnormal{MGL}}(\SH(S)) \simeq \SH^{\fsyn}(S),\]
where \(\SH^{\fsyn}(S)\) is the analogue of \(\SH^{\fr}(S)\) using \(\Cor^{\fsyn}(\Sm)\) instead of \(\Cor^{\fr}(\Sm)\).
\end{theorem}

\begin{theorem}\label{thm:B2}
If \(k\) is a perfect field, then
\[\textnormal{Mod}_{\textnormal{MGL}}(\SH(S)^{\veff}) \simeq \mathcal{H}^{\fsyn}(k)^{\gp}.\]
\end{theorem}
This theorem is a combination of \Cref{thm:B1} with a cancellation theorem in \(\mathcal{H}^{\fsyn}(k)^{\gp}\).

\subsection{Framed models for effective Thom spectra}\label{subsec:ThomSpec}
\begin{definition}
Consider a smooth \(S\)-scheme \(X\) and \(\xi \in K(X)_{\geq 0}\). We define a framed presheaf 
\[(X, \xi)^{\fr} \in \mathcal{P}_{\Sigma}(\Cor^{\fr}(\Sm))\] by sending a scheme \(U\) to the space of spans \((U \overset{f}{\leftarrow} Z \overset{g}{\rightarrow} X)\)
where \(f\) is finite and quasi-smooth together with \(\alpha \colon T_f \simeq -g^*(\xi) \in K(Z).\)
\end{definition}

This definition is not enough to define a presheaf on \(\Cor^{\fr}(\Sm)\); there is some work involved in making this precise which can be found in \cite[Appendix B]{EHKSYCobordism}.

\begin{remark}
If \(\text{rank}(\xi) =0,\) then \(Z\) is underived. E.g. \[(X,0)^{\fr} = \Cor_S^{\fr}(-, X).\]
\end{remark}

These presheaves turn out to be framed models for effective Thom spectra.
\begin{theorem}[{\cite[Theorem 3.2.1]{EHKSYCobordism}}]
\label{thm:FramedThomSpectra}
There is an equivalence
\[Th_X(\xi) \simeq \Sigma^{\infty}_{\fr}(X, \xi)^{\fr}\]
in \(\SH(S) \simeq \SH^{\fr}(S).\)
\end{theorem}
The proof consists of two parts: first constructing a map between these two objects, and then showing that it is an equivalence.
\begin{proof}
Suppose that \(\xi\) is a vector bundle \(V\) on \(X\), and let \(V^\times\) denote the complement of the zero section inside \(V\). We define a map as follows:

Consider the diagram
\begin{center}
\begin{tikzcd}
\emptyset \arrow[d] \arrow[rr] && X \arrow[dl, "z" above left, hook'] \arrow[dr, equal] & \\
V^\times \arrow[r, hook]&V & & X,
\end{tikzcd}
\end{center}
where \(z\) is the zero section.
The canonical equivalence \(T_z \simeq -V\) defines a map 
\[V \longrightarrow (X, V)^{\fr},\]
and we get a diagram
\begin{equation}\label{eq:ThomSpec}
    \begin{tikzcd}
    (V^\times)^{\fr} \arrow[d, hook] \arrow[dr, "0"] \\
    V^{\fr} \arrow[r] \arrow[d]& (X, V)^{\fr} \\
    Th_X(V)^{\fr} \arrow[ur, dotted],
    \end{tikzcd}
\end{equation}
where the dotted arrow is the comparison map we were looking for. 

One can show that the map is a motivic equivalence by reducing to the case where the base is a field and using a moving lemma from \cite{GNP}.

The general case for virtual tangent bundles is now mostly formal. The main point is to use the identification \(K(-) \simeq L_{\mathrm{Zar}} Vect(-)^{\gp}\). The Thom spectrum functor \(Th_X\from Vect(X) \to \SH(X)\) takes the direct sum of vector bundles to the tensor product of spectra and lands in the Picard \(\infty\)-groupoid \(\text{Pic}(\SH(X))\), and thus factors uniquely through the group completion of \(Vect(X)\). 

We can then formally extend \eqref{eq:ThomSpec} to a map
\[Th_X(\xi)^{\fr} \longrightarrow (X, \xi)^{\fr},\]
which is natural in the virtual vector bundle \(\xi\) over \(X\).

To see that this map is an equivalence we use that virtual vector bundles of rank \(\geq 0\) are actual vector bundles Zariski locally.
\end{proof}

We state a result about K-theory that is needed in the proof of \Cref{thm:A1}.
\begin{proposition}[Bhatt--Lurie, {\cite[A.0.6]{EHKSYCobordism}}]\label{prop:BL}
Let \(R\) be a commutative ring. Then the \(K\)-theory functor
\[K \from CAlg^{\text{der}}_R \to \Spc\]
is left Kan extended from \(CAlg^{\text{Sm}}_R.\)
\end{proposition}
\begin{proof}
The formula for left Kan extension is
\[LKE(K \vert CAlg^{Sm}_R)(T) = \underset{S \to T}{\text{colim }} K(S),\]
where \(S \in CAlg^{Sm}_R\). The index category \(\{S \to T\}\) has finite coproducts given by tensor products, which make it sifted. 

By the identity
\[K = Vect^{\gp}\]
and the fact that the group completion commutes with sifted colimits, we are reduced to show that \(Vect(-)\) is left Kan extended from \(CAlg^{Sm}_R\).
Using that \( Vect = \underset{n}{\text{colim }} Vect_{\leq n}\) we reduce to showing that each \(Vect_{\leq n}\) is left Kan extended.
One way of showing this is by taking the Jouanolou device
\[U_n \longrightarrow Gr_n,\]
which is an affine bundle whose total space is affine. Then there is a map
\[\coprod_{k \leq n} U_k \overset{f}{\longrightarrow} Vect_{\leq n}.\]
The map \(f\) is a section-wise surjection because we are dealing with affine schemes, so that vector bundles are generated by their global sections. Hence it may be lifted to a point in the Grassmannian, which further can be lifted to the Jouanolou device since it is an affine bundle over an affine scheme.

We therefore have
\[Vect_{\leq n} \simeq \text{colim } \check{C}_{\bullet}(f),\]
where \(\check{C}_{\bullet}(f) = \Delta^{op} \to \text{Ind}(\text{smooth affine schemes})\). The same remains true of all the self-intersections of \(\coprod_{k \leq n} U_k\), since \(Vect_{\leq n}\) is a stack with smooth  affine diagonals. Thus the iterated fiber product of \(\coprod_{k \leq n} U_k\) is still a colimit of smooth and affine schemes. Combining the colimits we thereby find that the presheaf \(Vect\) is a colimit of smooth affine \(R\)-schemes, which is precisely what it means to be left Kan extended from \(CAlg^{Sm}_R\). 
\end{proof}

Now for the proof of \Cref{thm:A1}.
\begin{proof}[Proof of \Cref{thm:A1}]
For \(n=0\) we can write
\[\textnormal{MGL} \simeq \underset{X,\xi}{\text{colim }} Th_X(\xi),\]
by \cite[Section 16]{BH},
where the colimit is taken over \(X \in \Sm\) and \(\xi \in K(X)\) with \(\text{rank}(\xi)=0\). 
Hence
\[\textnormal{MGL} \simeq \Sigma^{\infty}_{\fr} \underset{X, \xi}{\text{colim }} (X, \xi)^{\fr}\]
by \Cref{thm:1}.
Further, there is a forgetful map
\begin{center}
    \begin{tikzcd}
    \underset{X, \xi}{\text{colim }} (X, \xi)^{\fr} \arrow[d] \\
    FSyn
    \end{tikzcd}
\end{center}
by using that an object in \(\underset{X, \xi}{\text{colim }} (X, \xi)^{\fr}\) consists of a span \(U \overset{f}{\leftarrow} Z \overset{g}{\to} X\) 
which we forgetfully map to \(U \overset{f}{\leftarrow} Z.\) We claim that this map in an equivalence.

By taking the fiber one obtains a Cartesian square
\begin{center}
    \begin{tikzcd}
    \underset{X, \xi}{\text{colim }} (X, \xi)^{\fr} \arrow[d]  & \arrow[l] \arrow[d] \underset{X, \xi}{\text{colim }} \{ Z \overset{g}{\to} X,  T_f \simeq -g^*(\xi)\}\\
    FSyn & \arrow[l] \{U \overset{f}{\leftarrow} Z\}. 
    \end{tikzcd}
\end{center}
Formally rewriting we find that
\begin{align*}
\underset{X, \xi}{\text{colim }} \{ Z \overset{g}{\to} X,  T_f \simeq -g^*(\xi)\}
&\simeq
\underset{X, \xi,g}{\text{colim }} \{ T_f \simeq -g^*(\xi)\}\\
&\simeq
\underset{X, g}{\text{colim }} \text{fib}_{T_f}\left( K(X) \overset{g^*}{\rightarrow} K(Z) \right),
\end{align*}
and exchanging the fiber and colimit we get
\[\underset{X, g}{\text{colim }} \text{fib}_{T_f}\left( K(X) \overset{g^*}{\rightarrow} K(Z) \right) \simeq \text{fib}_{T_f}\left( \underset{X, g}{\text{colim }} K(X) \rightarrow K(Z) \right).\]
This we recognize as the formula for the left Kan extension, i.e.
\[\underset{X, g}{\text{colim }} K(X) \simeq LKE(K \vert \Sm)(Z). \]

It remains to show that
\[\underset{X, \xi}{\text{colim }} \{ Z \overset{g}{\to} X,  T_f \simeq -g^*(\xi)\}\]
is contractible, which we have reduced to showing that
\[\text{fib}_{T_f}\Big(LKE(K \vert \Sm)(Z) \to K(Z)\Big)\]
is contractible. 
However, it follows from \cref{prop:BL} that the map 
\[LKE(K \vert \Sm)(Z) \to K(Z)\]
is an equivalence, hence the fiber must be trivial. 

The proof for 
\(\Sigma^{n}_T \textnormal{MGL}\)
for \(n>0\) is similar, but with \(\xi \in K(X)\) such that \(\text{rank}(\xi)= n\).
\end{proof}

\begin{question}\label{OpenQuestion}
Is \Cref{thm:A1} still true if one relaxes the ``finiteness'' condition to a ``properness'' condition?
\end{question}

Note that the finiteness condition was not necessary to get a cobordism class. If one drops this condition, then \(FSyn\) becomes equal to \(PQSm^0\), which is a much larger stack. To see this one can start with any smooth proper scheme of any dimension \(n\) and find a derived scheme structure on it which is quasi-smooth and of dimension 0; we may for instance take the vanishing locus of the zero function \(n\) times.

It turns out that the naive cobordisms space \(L_{\A^1} PQSm^0\) is a group under disjoint union, which is not true for \(L_{\A^1}FSyn\). There should then be maps
\begin{equation}
    \begin{tikzcd}[column sep=huge]
    FSyn \arrow[r, hook] & PQSm^0 \arrow[r, "\text{conjecturally}"] & \Omega^{\infty}_T \textnormal{MGL},
    \end{tikzcd}
\end{equation}
where the last map uses the transfers in \(\textnormal{MGL}\)-cohomology. This would imply that 
\(\Sigma^{\infty}_{\fr} PQSm^0\)
has \(\textnormal{MGL}\) as a retract. Over perfect fields one can therefore conclude that \(\Omega^{\infty}\textnormal{MGL}\) is a direct factor of \(L_{\mathrm{Zar}} L_{\A^1} PQSm^0\). It is an open problem whether they in fact are equivalent.

\printbibliography

\noindent\textsc{Fakult\"{a}t f\"{u}r Mathematik, Universit\"{a}t Regensburg, Universit\"{a}tsstr. 31, 93040 Regensburg, Germany}\\
\textit{Email adress: }
\texttt{marc.hoyois@ur.de} \\
\textit{URL: } \url{http://www.mathematik.ur.de/hoyois/}

\vspace{5mm}

\noindent\textsc{Department of Mathematics, University of Oslo, Moltke Moes vei 35, 0851 Oslo, Norway}\\
\textit{Email adress: }
\texttt{ntmarti@math.uio.no} \\
\textit{URL: } \url{http://www.nikolaiopdan.com}

\end{document}